\documentclass[11pt]{amsart}
\usepackage{amsmath,amssymb, amscd}
\usepackage{graphicx}

\def\sqr#1#2{{\vcenter{\hrule height.#2pt
        \hbox{\vrule width.#2pt height#1pt \kern#1pt
                \vrule width.#2pt}
        \hrule height.#2pt}}}

\newtheorem{theorem}{Theorem}[section]

\newtheorem{proposition}[theorem]{Proposition}
\newtheorem{corollary}[theorem]{Corollary}

\theoremstyle{definition}
\newtheorem{definition}[theorem]{Definition}
\newtheorem{example}[theorem]{Example}

\newtheorem{question}[theorem]{Question}
\newtheorem{remark}[theorem]{Remark}
\newtheorem{setting}[theorem]{Setting}

\newtheorem{claim}[theorem]{Claim}

\DeclareMathOperator{\n}{\mathbf n}
\DeclareMathOperator{\m}{\mathbf m}
\DeclareMathOperator{\p}{\mathbf p}
\DeclareMathOperator{\q}{\mathbf q}

\DeclareMathOperator{\Spec}{Spec}

\DeclareMathOperator\ord{ord}

\DeclareMathOperator\tr{tr}

\DeclareMathOperator\Min{Min}

\DeclareMathOperator{\hgt}{ht}

\DeclareMathOperator{\Bl}{Bl}
\DeclareMathOperator{\Rees}{Rees }

\def\alert#1{\smallskip{\hskip\parindent\vrule%
\vbox{\advance\hsize-2\parindent\hrule\smallskip\parindent.4\parindent%
\narrower\noindent#1\smallskip\hrule}\vrule\hfill}\smallskip}

\begin{document}

\title[Rees valuations of  complete ideals]
{The Rees valuations of  complete ideals \\ in a regular local ring }

%    Information for first author

%\author{William Heinzer and Mee-Kyoung Kim}

\author{William Heinzer}
\address{Department of Mathematics, Purdue University, West
Lafayette, Indiana 47907 U.S.A.}
\email{heinzer@math.purdue.edu}

\author{Mee-Kyoung Kim}
\address{Department of Mathematics, Sungkyunkwan University, Jangangu Suwon
440-746, Korea}
\email{mkkim@skku.edu}

\date \today

\subjclass{Primary: 13A30, 13C05; Secondary: 13E05, 13H15}
\keywords{ Rees valuation, finitely supported ideal, special $*$-simple 
complete ideal, base points, point basis, transform of an ideal, 
local quadratic transform, projective equivalence, projectively full.
}

\maketitle
\bigskip

\begin{abstract}  Let $I$ be a complete $\m$-primary ideal of a  regular local 
ring  $(R,\m)$ of dimension $d \ge 2$. In the case of dimension two, the beautiful 
theory developed by Zariski implies that $I$ factors uniquely as a product of powers of 
simple complete ideals and  each of the simple complete factors of $I$ has a 
unique Rees valuation. In the higher dimensional case, a simple complete ideal of $R$ often
has more than one Rees valuation, and a complete $\m$-primary ideal $I$ may
have finitely many or infinitely many base points.  For the ideals having finitely many 
base points Lipman proves a unique factorization involving special $*$-simple complete
ideals and possibly negative exponents of the factors. Let $T$ be an infinitely near point to
$R$ with $\dim R = \dim T$ and $R/\m = T/\m_T$. We prove that the special $*$-simple 
complete ideal $P_{RT}$ has a unique Rees valuation if and only if either $\dim R = 2$ or 
there is no change of direction in the unique finite sequence of local quadratic 
transformations from $R$ to $T$. We also examine conditions for a  
complete ideal to be projectively full.  
\end{abstract}

\baselineskip 18 pt

\section{Introduction} \label{c1}

Motivation for our work in this paper comes from an interesting article of   Joseph Lipman  
\cite{L}.  Lipman considers the structure of a certain class of 
complete ideals, the finitely supported complete ideals,  in a regular local ring  (RLR) 
of dimension $d \ge 2$.  He  proves a factorization theorem for the finitely
supported complete ideals that  extends the factorization theory of complete ideals 
in a two-dimensional RLR    as developed 
by Zariski  \cite[Appendix 5]{ZS2}.  Other work on this topic has been done by John Gately
in \cite{G1} and \cite{G2},  and by Campillo, Gonzalez-Sprinberg and Lejeune-Jalabert in \cite{CGL}.

All rings we consider are assumed to be commutative with an identity element.
We use the concept of complete ideals as defined and discussed in Swanson-Huneke 
\cite[Chapters~5,6,14]{SH}.  We also use a number of concepts considered in Lipman's 
paper \cite{L}. The product of two complete ideals in a two-dimensional regular local 
ring is again complete. This no longer holds in higher dimension, \cite{C} or \cite{Hu}.
To consider the higher dimensional case,  one defines for ideals $I$ and $J$ the $*$-product,
$I*J$ to be the completion of $IJ$.  A complete ideal $I$ in a commutative ring $R$ 
is said to be {\bf $*$-simple} if $I \ne R$ and  if 
$I = J*L$ with ideals $J$ and $L$ in $R$  implies that either $J = R$ or $L = R$.

Another concept used by Zariski in \cite{ZS2} is 
that of the transform of an ideal; the complete transform of an
ideal is used in \cite{L} and \cite{G2}.  

\begin{definition} \label{1.1}
Let $R \subseteq T$ be unique factorization domains (UFDs)  with  $R$
and $T$ having the same field of fractions, and let  $I$ be an ideal of $R$ not contained in 
any proper principal ideal. 
\begin{enumerate}
\item The  {\bf transform } of $I$ in $T$ is the ideal $I^T = a^{-1}IT$,
where $aT$ is the smallest principal ideal in $T$ that contains $IT$.
\item  
The {\bf complete transform} of $I$ in $T$ is the completion $\overline {I^T}$ of $I^T$.
\end{enumerate} 
\end{definition}

A   proper  ideal $I$  in a commutative ring $R$   is {\bf simple} if $I \neq L\cdot H$, for any proper ideals $L$ and $H$.
An element $\alpha \in R$ is said to be {\bf integral over} $I$ 
if $\alpha$ satisfies an equation of the form
$$
\alpha^n + r_1 \alpha^{n-1} + \cdots + r_n = 0, \quad \text{where} \quad r_i \in I^i.
$$
The set of all elements in    $R$  that  are integral over an ideal $I$ forms 
an ideal, denoted by $\overline{I}$ and called the {\bf integral closure} of $I$.
An ideal $I$ is said to be {\bf complete}   (or, {\bf  integrally closed})   if $I = \overline {I}$.

For an ideal $I$ of a local ring $(R,\m)$, the {\bf order} of $I$, denoted $\ord_{R} I$, is $r$ if $I \subseteq \m^r$ but $I \nsubseteq \m^{r+1}$.
 If $(R,\m)$ is a regular
local ring, the function that associates to an element $a \in R$,   
the order of the principal ideal $aR$,   defines  a 
discrete rank-one valuation, denoted $\ord_R$  on the field of fractions of $R$.  The associated valuation 
ring (DVR)  is called {\bf the order valuation ring } 
of $R$.

Let $I$ be a nonzero ideal of a Noetherian integral domain $R$. 
The set  of  {\bf Rees valuation rings } of $I$  is  denoted {\bf Rees}$~I$, or by  {\bf Rees}$_RI$
to  also indicate the ring in which $I$ is an ideal. It 
is by  definition the set of DVRs
$$ 
\Big\{\Big(\overline {R\Big[ \frac{I}{a}\Big]}\Big)_{Q}~~ \vert \quad  0 \neq a \in I  \quad 
\text{and} \quad  Q \in \Spec \Big( \overline {R\Big[ \frac{I}{a}\Big]} \Big)\quad 
\text{is of height one with}~~
 I \subset Q \Big\},
$$
where $\overline{\cdot}$ denotes integral closure in the field of fractions. 
The corresponding  discrete valuations with value group $\mathbb Z$ 
are called the {\bf Rees valuations} of $I$. 
In general, if $J \subseteq I$ is a reduction, then we have $\Rees J = \Rees I$. 

An ideal $I$ is said to be {\bf normal} if all the powers of $I$ are complete.  
Let $I$ be a normal $\m$-primary ideal of a normal Noetherian  local 
domain $(R,\m)$.  The minimal prime ideals of $\m R[It]$
in the Rees algebra $R[It]$   are in one-to-one correspondence with the Rees valuation rings
of $I$.  The correspondence 
associates to each Rees valuation ring $V$ of $I$ a unique prime $P   \in \Min(\m R[It])$ such that $V=R[It]_{P} \cap \mathcal{Q}(R)$.  Properties of the quotient ring $R[It]/P$ relate to properties of certain birational extensions of $R$.

If $(R, \m)$ is a two-dimensional regular local ring,  then the  Zariski theory implies that a simple complete $\m$-primary ideal
has a unique Rees valuation ring. However, if the dimension of $R$ is greater than two, then a simple
complete $\m$-primary ideal  may have more than one Rees valuation ring; indeed, 
this is often the case even  for a special $*$-simple complete ideal as in Definition~\ref{2.13}. 
An ideal $I$ of a 
Noetherian integral domain $R$ is said to {\bf one-fibered} if $I$ has a unique Rees valuation.

In the case where $(R,\m)$ is a two-dimensional regular local ring, 
Zariski's unique factorization theorem implies that  a complete $\m$-primary ideal 
$I$ can be factored uniquely as a finite product of powers of simple complete ideals. 
  The  distinct simple factors of $I$ are in one-to-one 
correspondence  with the Rees valuation rings of $I$.

 If $I$ is a simple
complete ideal of a two-dimensional RLR and $R/\m$ is algebraically closed,
 Huneke and Sally \cite[Theorem~3.8]{HS} prove that $R[It]/P$ is regular.   This result is extended 
in \cite[Theorem~3.1]{K} by proving that if $R/\m$ is algebraically closed, then $I$ is a product 
of distinct simple complete ideals if and only if $R[It]/P$ is regular for each minimal prime $P$ of 
$\m R[It]$.

Let $(R,\m)$ be a regular local ring of dimension $d \ge 2$.  In Section~\ref{c2} we
discuss the structure of regular local rings $T$ birational over $R$, and the order 
valuation ring of $T$. In Section~\ref{s3} we review Lipman's unique factorization theorem
and raise several questions about the base points of finitely supported complete ideals.
Let $I$ be an $\m$-primary ideal of $R$. In Section~\ref{c3} we compare the Rees valuations of
$I$ with the Rees valuations of the transform $I_1$ of $I$ in $S_1 = R[\frac{\m}{x}]$,
where $x \in \m \setminus \m^2$.  We prove in Proposition~\ref{2.3} that 
$\Rees I \subseteq \Rees_{S_1} I_1 \cup \Rees \m$.  If $I$ is finitely supported, we
prove in Proposition~\ref{2.5} that $\Rees_{S_1} I_1 \subseteq \Rees I$, and demonstrate
in Example~\ref{2.61} that this may fail if $I$ is not finitely supported.

We observe in Remark~\ref{4.311} that every special $*$-simple complete ideal is
projectively full. In Proposition~\ref{4.7} we prove that a complete $\m$-primary 
ideal of $R$ is projectively full if 
the transform $I_1$ of
$I$ in $S_1$ is projectively full.   In Section~\ref{c4} we examine the structure of 
special $*$-simple complete ideals in terms of their Rees valuations.  
Let $T$ be an infinitely near point to
$R$ with $\dim R = \dim T$ and $R/\m = T/\m_T$. We prove in Theorem~\ref{4.40} 
that the special $*$-simple 
complete ideal $P_{RT}$ has a unique Rees valuation if and only if either $\dim R = 2$ or 
there is no change of direction in the unique finite sequence of local quadratic 
transformations from $R$ to $T$.    In the case where $T = R_1$ is a first local 
quadratic transform of $R$ and $R/\m \ne T/\m_T$, we demonstrate in 
Examples~\ref{4.121} and \ref{4.122} that sometimes the special $*$-simple complete 
ideal $P_{RT}$ has two Rees valuations and sometimes only one Rees valuation.
Examples~\ref{4.3} and \ref{4.4} illustrate a pattern where from $R_0$ to $R_2$ or
from $R_0$ to $R_3$ there is exactly one or exactly two changes of direction.

\section{Preliminaries } \label{c2}

Let $V$ be a valuation domain and let  $R$ be a subring of $V$. 
Let $\m(V)$ denote the unique maximal ideal of $V$.
We call the prime ideal $\m(V)\cap R$ of $R$ {\bf the center} of $V$ on $R$.

Let $(R,\m)$ be a Noetherian local domain with field of fractions $\mathcal Q(R)$.
A valuation domain  $(V, \m(V))$ is said to {\bf birationally dominate}  $R$ 
if $R \subseteq V \subseteq \mathcal Q(R)$ and  $\m(V) \cap R = \m$, that is, $\m$ is the 
center of $V$ on $R$.
The valuation domain  $V$ is said
to be a {\bf prime divisor} of $R$ if  $V$ birationally dominates $R$ 
and the transcendence degree of the field $V/\m(V)$ 
over $R/\m$ is $\dim R - 1$. If $V$ is a prime divisor of $R$, then $V$ is a DVR
\cite[p. 330]{A}.

The {\bf quadratic dilatation}  or {\bf blowup}  of $\m$ along $V$,  
cf.  \cite[ page~141]{N},  is the
unique local ring on the blowup $\Bl_{\m}(R)$  of $\m$ that is dominated by  $V$.  
The ideal $\m V$ is 
principal and is generated by an element of $\m$.  Let $a \in \m$ be such that $aV = \m V$.  Then
$R[\m/a] \subset V$. Let $Q := \m(V) \cap R[\m/a]$. Then $R[\m/a]_Q$ is the  { \bf  quadratic 
transformation   of }
$R$ {\bf along }  $V$.  In the special case where $(R,\m)$ is a $d$-dimensional regular local domain 
we use the following terminology.

\begin{definition}\label{2.1}
Let $d$ be a positive integer and let $(R, \m, k)$ be a $d$-dimensional regular local ring 
with maximal ideal $\m$ and  residue field $k$. 
Let $x\in \m \setminus \m^2$ and let $S_1 := R[\frac{\m}{x}]$. The ring  $S_1$ is a
$d$-dimensional regular ring in the sense that each localization of $S_1$ at a 
prime ideal is a regular local ring.  To see this, observe that  $S_1/xS_1$ is isomorphic to a 
polynomial ring in $d-1$ variables over the field $k$, cf. \cite[Corollary~5.5.9]{SH},
 and $S_1[1/x] = R[1/x]$ is a regular ring. Moreover,
$S_1$ is a UFD  since $x$ is a prime element of $S_1$ and
$S_1[1/x] = R[1/x]$ is a UFD, cf. \cite[Theorem~20.2]{M}.
Let $I$ is an $\m$-primary ideal of $R$ with $r:=\ord_{R}(I)$. Then one has in $S_1$
$$
IS_1=x^rI_1 \quad \text{for some ideal}\quad I_1 \quad \text{of}\quad S_1.
$$
We observe in Remark~\ref{2.11} that either $I_1 = S_1$ or $\hgt I_1 \ge 2$. 
Thus  $I_1$ is  the transform $I^{S_1}$  of $I$ in $S_1$ as
in Definiton~\ref{1.1}. 

Let $\p$ be a prime ideal of $R[\frac{\m}{x}]$ with $\m \subseteq \p$. 
      The local ring 
$$R_1:~= ~R[\frac{\m}{x}]_{\p}  ~ = (S_1)_{\p}
$$
 is called a {\bf local quadratic transform} of $R$;  the ideal
$I_1R_1$ is the  transform of $I$ in $R_1$ as in Definition~\ref{1.1}.  
\end{definition}

\begin{remark} \label{2.11} 
With the notation of Definition~\ref{2.1},  to justify that 
the ideal  $I_1$ is the transform of $I$ in $S_1$, we observe that the ideal $I_1$ is 
not contained in any height-one prime of $S_1$. For if $I_1  \subseteq xS_1$, then 
we would have $I \subseteq x^{r+1}S_1 \cap R = \m^{r+1}$, a contradiction to the
choice of $r$. If  $I_1 \subseteq q$, where $q$ is a height-one prime of $S_1$ 
different from $xS_1$, then  $I \subseteq q \cap R$. This is impossible 
since  $q \cap R$ is a height-one prime of $R$ and  $I$ is $\m$-primary.
\end{remark}

We follow the notation of \cite{L} and  refer to regular local rings of dimension at
least two 
as   {\bf points}.
A point $T$ is said to be {\bf infinitely near} to a point $R$,  in symbols, $R~ \prec~ T$, 
if there is a finite sequence of local quadratic transformations 
\begin{equation} \label{e1} 
R=:R_0 ~\subset~ R_1 ~\subset~ R_2~ \subset \cdots \subset~ R_n=T \quad (n \geq 0),
\end{equation}
where $R_{i+1}$ is a local quadratic transform of $R_i$ for $i=0,1, \ldots, n-1$.
If such a sequence of local quadratic transforms as in Equation~\ref{e1} exists, then it is unique and 
it is
called the {\bf quadratic sequence} from $R$ to $T$ \cite[Definition~1.6]{L}.

\begin{remark} \label{2.111} Let $(R,\m)$ be a regular local ring with $\dim R \ge 2$. 
As noted in \cite[Proposition~ 1.7]{L}, there is a one-to-one correspondence between the 
points $T$ infinitely near to $R$ and the prime divisors $V$ of $R$. This correspondence
is defined by associating with $T$ the order valuation ring $V$ of $T$.   Since $V$ is the unique  local quadratic transform of $T$
of dimension one,  the local quadratic sequence in Equation~\ref{e1}  extends to give 
Equation~\ref{e11}:
\begin{equation} \label{e11} 
R=:R_0 ~\subset~ R_1 ~\subset~ R_2~ \subset \cdots \subset~ R_n=T   ~\subset ~V.
\end{equation}
The one-to-one correspondence between the 
points $T$ infinitely near to $R$ and the prime divisors $V$ of $R$ implies that
$T$ is the unique point  infinitely near to $R$ for which the order valuation ring of $T$ is $V$.
However, if  $\dim R > 2$, then there  often exist 
regular local rings $S$ with $S \ne T$ such that $S$ birationally dominates $R$ 
and the order valuation ring of $S$ is $V$. We illustrate this in Example~\ref{2.112}.
\end{remark} 

\begin{example} \label{2.112} 
Let $(R,\m)$ be a 3-dimensional RLR with $\m = (x,y,z)R$, and let $V$ denote the
order valuation ring of $R$. Let  $S=R[\frac{y}{x}]_{(x, z)}$. 
Then $S$ is a $2$-dimensional RLR that birationally dominates $R$,  
and $V$ is the order valuation ring of $S$. Notice that $S$ is not infinitely near to $R$. 
\end{example} 

\begin{remark} \label{2.113}
Let $(R,\m)$ be a $d$-dimensional RLR with $d \ge 2$ and let $V$ be the order
valuation ring of $R$. Let $(S,\n)$ be a $d$-dimensional RLR that is a birational
extension of $R$. Then
\begin{enumerate}
\item $S$ dominates $R$.
\item If $V$ dominates $S$, then $R = S$.
\item  Thus $R$ is the unique $d$-dimensional RLR having order valuation ring $V$ among
the regular local rings birational over $R$.
\end{enumerate}
\end{remark}

\begin{proof}  For item (1), let $P := \n \cap R$. Then $R_P \subseteq S$. If $P \ne \m$, 
then $\dim R_P = n < d$. Since every  birational extension of an $n$-dimensional Noetherian
domain has dimension at most $n$, we must have $\dim S \le n$, a contradiction. Thus $S$
dominates $R$. Item (2) follows from \cite[Corollary~2.6]{Sa2}.  In more detail, if $V$ dominates
$S$, then $R/\m = S/\n$ and the elements in a minimal generating set for $\m$ are part of 
a minimal generating set for $\n$. Hence we have $\m S = \n$. By Zariski's Main Theorem as in
\cite[(37.4)]{N}, it follows that $R = S$. Item (3) follows from item (2).
\end{proof}

Example~\ref{2.114} demonstrates the existence of a prime divisor $V$ 
for a $3$-dimensional RLR  $(R, \m, k)$ for which there 
exist infinitely many distinct $3$-dimensional RLRs that 
birationally dominate $R$, and  have $V$ as their  order valuation ring.

\begin{example}\label{2.114}
Let $(R, \m, k)$ be a $3$-dimensional regular local ring with residue field $R/\m = k$
and maximal ideal $\m = (x,y,z)R$.   Let $V$ be the  prime divisor of $R$ 
corresponding to the valuation $v$,
where $v(x)= v(y)=1$ and $v(z)=3$, and where the images of 
$\frac{x^3}{z}$ and $\frac{y^3}{z}$ in the residue field  $k_v$ of $V$ are 
algebraically independent over $R/\m = k$.
Then we have :
\begin{enumerate}
\item 
In the unique finite sequence of local quadratic transformations given by 
\cite[Proposition~1.7]{L}, we have:
$$
R=:R_0~\subset~ R_1 ~\subset~ R_2~\subset ~V,
$$
 where
$$
\aligned
&R_1:=R[\frac{\m}{x}]_{\p},\quad \text{ where}\quad \p:=\m(V) \cap R[\frac{\m}{x}]=(x, \frac{z}{x})R[\frac{\m}{x}],\\
&R_2:=R_1[\frac{\m_1}{x}]_{\q},\quad \text{ where}\quad \q:=\m(V) \cap R_1[\frac{\m_1}{x}]=(x, \frac{z}{x^2})R_1[\frac{\m_1}{x}],
\endaligned
$$
and  $V$ is the order valuation ring of $R_2$.
Notice that $(R_1, \m_1)$ and $(R_2, \m_2)$ are $2$-dimensional RLRs.

\item
For each integer $n \ge 1$,  let $T_n:=R[\frac{z}{x^2+y^{2n+1}}]_{(x,y,\frac{z}{x^2+y^{2n+1}})}$. 
Then  we have: 
\begin{enumerate}
\item By \cite[Lemma~4.2]{Sa1}, each 
 $T_n$ is a $3$-dimensional RLR that birationally dominates $R$. 
\item  For each $n$, the images of $\frac{y}{x}$ and $\frac{z}{x^3}$ in $k_v$
are algebraically independent over $k$. Hence 
the order valuation ring of $T_n$ is $V$.  
\item 
By \cite[Corollary~4.5]{Sa1}), the elements in the family $\{T_n\}_{n=1}^\infty$ are 
distinct. 
\end{enumerate}
\end{enumerate}
\end{example}

\begin{definition} \label{2.12} 
A {\bf base point} of a nonzero ideal $I \subset R$ is a point $T$  infinitely near 
 to  $R$ such 
that $I^{T} \neq T$. The set of base points of $I$ is denoted by
$$
\mathcal{BP}(I)=\{~ T~\vert  ~T \text{ is a point such that}~~R \prec T ~\text{and}~\ord_T (I^T)\neq 0~\}.
$$
The {\bf point basis} of a nonzero ideal $I \subset R$ is the family of nonnegative integers
$$
\mathcal B(I)=\{~\ord_T (I^T) ~\vert~ R~ \prec~ T ~\}.
$$
The nonzero ideal  $I$ is said to be {\bf finitely supported} if  
$I$ has only  finitely many base points.
\end{definition}

\begin{definition} \label{2.13}
Let $R \prec T$ be points such that $\dim R=\dim T$. Lipman proves in \cite[Proposition~2.1]{L}
the existence of a unique complete ideal $P_{RT}$ in $R$ such that for every point $A$
with $R \prec A$,  the 
complete transform 
$$
\overline{(P_{RT})^A} ~ \text{ is }  ~
\begin{cases}
\text{ a $*$-simple ideal if  } ~ A \prec T, \\ \text{ the ring } A \text{ otherwise.} 
\end{cases}  
$$
The ideal $P_{RT}$ of $R$ is said to be  a {\bf special $*$-simple complete ideal}. 

In the case where $R \prec T$ and  $\dim R = \dim T$, we say that 
the order valuation ring of $T$
is a {\bf special prime divisor} of $R$.
\end{definition} 

\begin{remark} \label{2.131}
With notation at in Definition~\ref{2.13},  a prime divisor $V$ of $R$ is 
special if and only if the unique point $T$ with $R \prec T$ such that the
order valuation ring of $T$ is $V$ has $\dim T = \dim R$. Let $\dim R = d$. 
If  $V$ is a special prime divisor of $R$, then the residue field of $V$ is 
a pure transcendental extension of degree $d - 1$ of the residue field $T/\m(T)$
of $T$, and $T/\m(T)$ is  
a finite algebraic
extension of $R/\m$.  If the residue field $R/\m$ of $R$ is 
algebraically closed and $V$ is a special prime divisor of $R$, then the 
residue field of $V$ is a pure transcendental extension of $R/\m$ of transcendence
degree $d - 1$.  

It would be interesting to  identify   and describe in other ways   
the special prime divisors of $R$ 
among the set of all prime divisors of $R$.
\end{remark}

\section{Factorization as products of special $*$-simple complete ideals} \label{s3}

Let $R = \alpha$ be a $d$-dimensional regular local ring with $d \ge 2$.
 Lipman in \cite[Theorem~2.5]{L}
proves that for every finitely supported complete ideal $I$ of $R$ there exists a 
unique family of integers 
$$
(n_\beta) ~ = ~ (n_\beta(I))_{\beta \succ \alpha, ~ \dim \beta = \dim \alpha}
$$
such that $n_\beta = 0$ for all but finitely many $\beta$  and such that 
\begin{equation} \label{es30}
\Big( \prod_{n_\delta < 0} P_{\alpha\delta}^{-n_\delta} \Big)*I ~=~ 
\prod_{n_\gamma > 0}P_{\alpha\gamma}^{n_\gamma}
\end{equation}
where $P_{\alpha\beta}$ is the special $*$-simple ideal associated with $\alpha \prec \beta$
and the products are $*$-products.  The product 
on the left in Equation~\ref{es30}  is  over all $\delta \succ \alpha$ 
such that $n_\delta < 0$ and the product on 
the right is over all $\gamma \succ \alpha$ such that $n_\gamma > 0$.

Lipman gives the following example to illustrate this decomposition.

\begin{example} \label{s3.1} Let $k$ be a field and let $\alpha = R =  k[[x,y,z]]$ be the
formal power series ring in the 3 variables $x,y,z$ over $k$.  Let 
$$
\beta_x = R[\frac{y}{x}, \frac{z}{x}]_{(x, y/x, z/x)},  \quad 
\beta_y = R[\frac{x}{y}, \frac{z}{y}]_{(y, x/y, z/y)},   \quad 
\beta_z = R[\frac{x}{z}, \frac{y}{z}]_{(z, x/z, y/z)} 
$$ 
be the local quadratic transformations of $R$ in the $x,y,z$ directions. The associated 
special $*$-simple ideals are
$$
P_{\alpha \beta_x} = (x^2, y, z)R, ~ \quad P_{\alpha \beta_y} = (x,y^2, z)R, ~ \quad
P_{\alpha \beta_z} = (x,y,z^2)R.
$$
The equation
\begin{equation} \label{es1}
(x,y,z)(x^3, y^3, z^3, xy, xz, yz) ~= ~ P_{\alpha \beta_x} P_{\alpha \beta_y} P_{\alpha \beta_z}
\end{equation}
represents the factorization of the finitely supported ideal $I = (x^3, y^3, z^3, xy, xz, yz)R$
as a product of special $*$-simple ideals. Here $P_{\alpha \alpha} = (x,y,z)R$. The base points
of $I$ are 
$\mathcal {BP}(I) = \{\alpha, \beta_x, \beta_y, \beta_z\}$ and the point basis of $I$ is
$\mathcal B(I) = \{2,1,1,1\}$. Equation~\ref{es1}  represents the following equality of 
point bases
$$
\mathcal B(P_{\alpha\alpha})~ + ~ \mathcal B(I) ~=~ \mathcal B(P_{\alpha \beta_x})
~+ ~ \mathcal B( P_{\alpha \beta_y}) ~+~ \mathcal B( P_{\alpha \beta_z}).
$$
Each of $P_{\alpha \beta_x},  P_{\alpha \beta_y},  P_{\alpha \beta_z}$ has a unique Rees
valuation. Their product has in addition the order valuation of $\alpha$ as a Rees valuation. 
\end{example}

\begin{question}  \label{s3.3}
Let $I$ be a finitely supported ideal of a regular local ring $R$.
\begin{enumerate}
\item If the base points of $I$ are linearly ordered, does it
follow that $I$ is a $*$-product of special $*$-simple complete ideals, i.e., in
the factorization given in Equation~\ref{es30} are all the integers $n_\beta$ nonnegative?  
\item  If $I$ is $*$-simple and if the base points of $I$ are linearly ordered, does 
it follow that $I$ is a special $*$-simple ideal? 
\item If $R \prec T$ with $\dim R = \dim T$ and $R \ne T$, 
can it happen that some power of the special $*$-simple
complete ideal 
$P_{RT}$  has the maximal ideal $\m$ of $R$ as a factor, that is, 
can there exist an ideal $Q$ of $R$
such that $\m Q = (P_{RT})^n$ for some positive integer $n$? 
\footnote{In a joint paper with Matthew Toeniskoetter titled ``Finitely supported
$*$-simple complete ideals in a regular local ring'', we have answered these 
three questions in
the case where $I$ is a finitely supported monomial ideal.}
\end{enumerate}
\end{question}

\section{Rees valuations of ideals of a regular local ring }\label{c3}

We use the following setting. 

\begin{setting}\label{2.2}
Let $d$ be a positive integer and let $(R, \m, k)$ be a $d$-dimensional regular local ring 
with maximal ideal $\m$ and  infinite   residue field $k$. 
Let $I$ be an $\m$-primary ideal. For each $V \in \Rees I$, 
let $v$ denote the corresponding 
Rees valuation with value group $\mathbb Z$.  
 Let $x\in \m \setminus \m^2$ be such that $xV = \m V$ for each $V \in \Rees I$. Since
the field $k$ is infinite and the set $\Rees I$ is finite, it is 
possible to choose such an element $x$. Let  
$r = \ord_R I$. As in Definition~\ref{2.1}, we have $IS_1 = x^rI_1$, where 
$I_1$ is  the transform of $I$ in $S_1 = R[\frac{\m}{x}]$. 
\end{setting}

\begin{remark} \label{2.25}
With the notation of Setting~\ref{2.2}, we have:
\begin{enumerate}
\item If $J \subseteq I$ 
is a reduction of $I$ in $R$, then $\ord_R J = \ord_R I = r$, and 
$JS_1  = x^rJ_1$ is a reduction of $IS_1 = x^rI_1$ in $S_1$. It follows
that  $J_1$  is the transform in $S_1$ of $J$ and $J_1 \subseteq I_1$ is a 
reduction of $I_1$ in $S_1$. 
\item If $J = (a_1, \ldots, a_d)R$ is a reduction of $I$ 
then a DVR $V$ that birationally
dominates $R$ is a Rees valuation ring of $I$ if and only if the images of 
$\frac{a_2}{a_1}, \ldots, \frac{a_d}{a_1}$ in the field $\frac{V}{\m(V)}$ are algebraically
independent over $k$. 
\item  The unique Rees valuation ring of $\m$ is 
$(S_1)_{xS_1}$,   i.e.,   
 $\Rees \m = \{ (S_1)_{xS_1} \}$.  

\end{enumerate} 
\end{remark}

\begin{proposition} \label{2.3} 
With the notation of Setting~\ref{2.2}, we have:
\begin{enumerate} 
\item
If $I_1 = S_1$, then $v = \ord_R$ and $\ord_R$ is the unique Rees valuation of $I$.
\item 
If $I_1 \ne S_1$ and $v \ne \ord_R$, then $V \in \Rees_{S_1} I_1$. 
\item
In general, we have $\Rees I \subseteq \Rees_{S_1} I_1 \cup \Rees \m$. 
\end{enumerate} 
\end{proposition}
\begin{proof}

For the proof of item $(1)$, let $\p:=\m(V) \cap S_1$ be the center of $V$ on $S_1$. Since
$xS_1 = \m S_1  \subseteq \p$ and $(S_1)_{xS_1}$ is the valuation ring of $\ord_R$, it 
suffices  to show that $\hgt \p=1$. 
By the Dimension Formula (\cite[page~119]{M}), we have 
$$
\hgt \p ~ = ~ 1 ~ \iff  ~  \tr.\deg_k \kappa \big(S_1/{\p} \big)~= ~d-1,
$$ 
where $\kappa(S_1/\p)$ denotes the field of fractions of $S_1/\p$. 
Let $J:=(a_1, \ldots, a_d)R$ be a reduction of $I$. Since
$V \in \Rees I = \Rees J$, the images of $a_2/a_1, \ldots, a_d/a_1$ in $V/\m(V) =: k_v$
are algebraically independent over $k$.
Since $JS_1$ is a reduction of $IS_1$ and $IS_1=x^r S_1$ is a principal ideal,
we have 
$$
J_1:=(f_1, \ldots, f_d)S_1=S_1, \quad \text{where}\quad f_i:=\frac{a_i}{x^r}\quad\text{for}\quad 1\leq i\leq d.
$$
It follows that  $(f_1, \ldots, f_d)V=V$ and thus $v(f_i)=0$ for  $i=1,\ldots, d$. 
Consider the inclusion maps:
$$
 \frac{S_1}{\p} ~ \hookrightarrow ~ \kappa\big(S_1/{\p}\big)~ \hookrightarrow ~ \frac{V}{\m(V)}.
$$
Since $v(f_i)=0$ and $\p=\m(V) \cap S_1$,  we have $f_i \in S_1 \setminus \p$. Therefore  
the images of $
\frac{f_2}{f_1} = \frac{a_2}{a_1}, \ldots, \frac{f_d}{f_1} = \frac{a_d}{a_1}$ 
in $\frac{V}{\m(V)}$ are in 
the subfield $\kappa\big(S_1/{\p}\big)$ of $\frac{V}{\m(V)}$. Hence  $\tr.\deg_k \kappa \big(S_1/{\p} \big)~= ~d-1$. 

For the proof of item $(2)$, we use the notation of the proof of item $(1)$. Notice that $f_1, 
\ldots, f_d$ all have the same $v$-value. Moreover, since $V \ne (S_1)_{xS_1}$, 
we must have $v(f_i) > 0$; for if $v(f_i) = 0$, the proof 
of item $(1)$ shows that $\hgt \p = 1$ and thus $V = (S_1)_{xS_1}$, a contradiction to
our assumption. Thus $J_1 = (f_1, \ldots, f_d)S_1 \subseteq \p$.  Since $J_1$ is a reduction
of $I_1$ and the images of $
\frac{f_2}{f_1} = \frac{a_2}{a_1}, \ldots, \frac{f_d}{f_1} = \frac{a_d}{a_1}$ in $\frac{V}{\m(V)}$
are algebraically independent over $k$, and thus  we have $V \in \Rees_{S_1} J_1 = \Rees_{S_1} I_1$. 

Item $(3)$ follows from  items (1) and (2).
\end{proof} 
 
\begin{proposition} \label{2.31} Let the notation be as in Setting~\ref{2.2} and
let $V \in \Rees I$.  As in Equation~\ref{e11}, there exists a unique 
finite sequence of local quadratic transforms 
\begin{equation} \label{e231} 
R=:R_0 ~\subset~ R_1 ~\subset~ R_2~ \subset \cdots \subset~ R_n=T   ~\subset ~V,
\end{equation} 
where $V$ is the order valuation ring of $T$. Then the  points $R_0, \ldots, R_n$ are 
all base points of $I$. 
\end{proposition}

\begin{proof}  If $V$ is the order valuation ring of $R$, then $n = 0$ in Equation~\ref{e231}
and $R_0$ is a base point of $I$. If $n > 0$, consider $S_1 = R[\frac{\m}{x}]$ as in 
Setting~\ref{2.2}. By Proposition~\ref{2.3}, $V \in \Rees_{S_1}I_1$. The local quadratic 
transform $R_1$ of $R$ is a localization of $S_1$ and $I_1R_1$ is a proper ideal in $R_1$. 
By an inductive argument on the length $n$ of the sequence in Equation~\ref{e231}, we
conclude that the points $R_1, \ldots, R_n$ are base points of $I_1R_1$ and therefore
also base points of $I$. 
\end{proof}

\begin{remark}\label{2.4}
Let the notation be as in  Setting~\ref{2.2}. 
If $I$ is  a finitely supported ideal in $R$, then  \cite[Corollary~1.22]{L} implies
that $\hgt I_1 = d$ and $\dim(S_1/I_1) = 0$. 
\end{remark}

\begin{proposition}\label{2.5}
Let the notation be as in  Setting~\ref{2.2}. 
If $I$ is  a finitely supported ideal in $R$, then 
$$ \Rees_{S_1} I_1 \subseteq \Rees I.$$
\end{proposition}

\begin{proof}
By Proposition~\ref{2.3}, we have $I_1 = S_1$ if and only if $\ord_R$ is the 
unique Rees valuation of $I$.  Assume $I_1 \ne S_1$, and  
let $J:=(a_1, \ldots, a_d)R$ be a reduction of $I$. 
Since $JS_1$ is a reduction of $IS_1$ and $I_1 \ne S_1$,
we have 
$$
J_1:=(f_1, \ldots, f_d)S_1 \ne S_1, \quad \text{where}\quad f_i:=\frac{a_i}{x^r}\quad\text{for}\quad 1\leq i\leq d.
$$
It follows that $J_1$ is a reduction of $I_1$. By  Remark~\ref{2.4}, we have $\hgt J_1 = d$.
Hence $\{f_1, \ldots, f_d \}$  is a minimal set of generators of $J_1$.  
 For each 
$W \in \Rees_{S_1} I_1 = \Rees_{S_1} J_1$, the images of $f_2/f_1, \ldots, f_d/f_1$ in $W/\m(W)$
are algebraically independent over $k$. 
Let $\frak q:=\m(W) \cap \overline{R\big[\frac{J}{a_1}\big]}$ be the 
center of $W$ on $A:=\overline{R\big[\frac{J}{a_1}\big]}$.
Since $\frac{f_2}{f_1} = \frac{a_2}{a_1}, \ldots, \frac{f_d}{f_1} = \frac{a_d}{a_1}$,
 we have that the images of
$\frac{a_2}{a_1}, \ldots, \frac{a_d}{a_1}$ in $\frac{W}{\m(W)}$ are in 
the subfield $\kappa\big(A/{\frak q}\big)$ of $\frac{W}{\m(W)}$. 
Hence  $\tr.\deg_k \kappa \big(A/{\frak q} \big)~= ~d-1$.
 By the Dimension Formula (\cite[page~119]{M}), $\hgt(\frak q) = 1$. Hence
$A_{\frak q} = W$. 
\end{proof}

Propositions~\ref{2.3} and ~\ref{2.5}  imply the following corollary.

\begin{corollary}\label{2.6}
Let the notation be as in  Setting~\ref{2.2}. 
If $I$ is  a finitely supported ideal in $R$ and $\ord_R$ is not a Rees valuation of $I$, then 
$$ \Rees_{S_1} I_1 = \Rees I.$$
\end{corollary}

We use Proposition~\ref{2.60} to demonstrate that without the assumption 
in Proposition~\ref{2.5} that the 
ideal $I$ has finite support, there sometimes  exist Rees valuations of the  transform $I_1$ of $I$ 
that are not Rees valuations of $I$.

\begin{proposition} \label{2.60} With the notation of Setting~\ref{2.2}, if each Rees 
valuation of $I$ is centered on a maximal ideal of $S_1$ and $\dim (S_1/I_1) > 0$, 
then there exist Rees valuations of $I_1$ that are not Rees valuations of $I$, i.e.,
$\Rees_{S_1} I_1 \nsubseteq \Rees I$.  
\end{proposition} 

\begin{proof} Since $\dim(S_1/I_1) > 0$, there exists a minimal prime $P$ of $I_1$ 
such that $P$ is not a maximal ideal of $S_1$.  Every minimal prime of $I_1$ is the
center of at least one Rees valuation ring of $I_1$. Let $V \in \Rees_{S_1} I_1$ be centered
on $P$. By assumption, $V \notin \Rees I$.  
\end{proof}

We present in 
Example~\ref{2.61} a specific example where the hypotheses of Proposition~\ref{2.60} hold. 
By Proposition~\ref{2.5}, the ideal $I$ of Example~\ref{2.61} is
not finitely supported. 

\begin{example}\label{2.61}
Let $(R,\m,k)$ be a three-dimensional regular local ring with residue field $R/\m = k$
and maximal ideal $\m = (x,y,z)R$.   
Let
$$
J:=(x^2, y^3, z^5)R, \quad \text{and}\quad S_1:=R\big[\frac{\m}{z}\big] \quad \text{with}\quad  x_1:=\frac{x}{z} \quad y_1:=\frac{y}{z}.
$$
The ideal  $I:= (x^2, y^3, z^5, xy^2, xyz^2, y^2z^2, yz^4)R$ is the integral closure of $J$, and :
\begin{enumerate}
\item   The ideals $J$ and $I$ have a unique Rees valuaton $v$, where 
$$v(x)=15,\quad v(y)=10,\quad\text{ and}\quad  v(z)=6,$$
and the images of $\frac{x^2}{z^5}$ and $\frac{y^3}{z^5}$ in $k_v$ are algebrically
independent over $k$.

\item The center of $v$ on $S_1$ is the maximal ideal $(x_1, y_1, z)S_1$.

\item  $J_1=(x_1^2, zy_1^3, z^3)S_1\subset I_1=(x_1^2, zy_1^3, z^3, x_1y_1^2z, x_1y_1z^2, y_1^2z^2)S_1$. 
We have $J_1$ is a reduction of $I_1$ with $\hgt J_1=2$ and $\mu (J_1)=3$.
\item  The ideal $I$ is not finitely supported. 
\item  $\Rees_{S_1} I_1=\{V, W \}$, where $V$ and $W$ denote the valuation rings corresponding
to $v$ and $w$, respectively, and  where 
$$w(x)=3,\quad w(y)=2,\quad\text{ and}\quad  w(z)=2,$$
and the images of $\frac{x^2}{y^3}$ and $\frac{x^2}{z^3}$ in $k_w$ are algebrically
independent over $k$.

\item  $\Rees_{S_1} I_1 \nsubseteq \Rees I$. 
\end{enumerate}
\end{example}
\begin{proof}   
The assertion in item (1) is well-known, see for example \cite[page~209]{SH}, 
and item~(2)  follows from item~(1). 
Since $I_1 \subseteq (x_1, z)S_1$, we have $\hgt I_1 = 2$ as asserted in item~(3). 
Item (4) follows from  
Remark~\ref{2.4}.
For the proof of item $(5)$, since $v$ is not $\ord_R$, Proposition~\ref{2.3} implies
that $V \in \Rees_{S_1} I_1$. We have $I_1 \subseteq \frak p:=(x_1, z)S_1$. Moreover :
\begin{enumerate}
\item   $y_1$ is a unit in the two-dimensional regular local ring $(S_1)_{\frak p}$. Also
 $\frak p \cap R = \m$ and the image of $y_1$ in the field of fractions of $S_1/\frak p$ is
algebrically independent over $R/\m$. 
\item  $(I_1)_{\frak p}=(x_1^2, z)$ is a simple complete ideal in $(S_1)_{\frak p}$.
\item  $\Rees (I_1)_{\frak p}=\{W\}$, where   $w(x_1)=1$, $w(z)=2$, and the image of 
$\frac{x_1^2}{z}$ in $k_w$ is algebraically independent over the field of fractions 
of  $S_1/\frak p$.
\item  $\Rees_{S_1} (I_1)=\{V, W\}$, where $w(x)=3$, $w(y)=z$, and $w(z)=2$.
\end{enumerate}
\end{proof}

In Example~\ref{2.62}, the height of the transform ideal $I_1$ is less than the height of $I$
and yet $\Rees I = \Rees_{S_1} I_1$.

\begin{example}\label{2.62}
Let $(R,\m,k)$ be a three-dimensional regular local ring with residue field $R/\m = k$
and and maximal ideal $\m = (x,y,z)R$.   
Let
$$
J:=(x^2, y^2, z)R, \quad \text{and}\quad S_1:=R\big[\frac{\m}{x}\big] \quad \text{with}\quad  y_1:=\frac{y}{x} \quad z_1:=\frac{z}{x}.
$$
The ideal $I: = (x^2, xy, y^2, z)R$ is the integral closure of $J$, and :
\begin{enumerate}
\item  The ideals $J$ and $I$ have  a unique Rees valuation $v$,  where 
$$
v(x)=1,\quad v(y)=1,\quad\text{ and}\quad  v(z)=2,
$$
and the images of $\frac{x^2}{z}$ and $\frac{y^2}{z}$ in $k_v$ are algebraically
independent over $k$.

\item  $J_1 = I_1 = (x, z_1)S_1$,  and hence $\hgt I_1=2$ and $\mu (I_1)=2$.
\item  $I$ is not finitely supported.
\item  $\Rees_{S_1}I_1=\Rees I$. 
\end{enumerate}
\end{example}
\begin{proof} Item~(1) is well-known, cf. \cite[Theorem~10.3.5]{SH}. Item~(2) is
clear and  item~(3) follows from  Remark~\ref{2.4}. 
For the proof of item $(4)$,
since $\frak p:=I_1$ is  a height two prime in $S_1$.
Then :
\begin{enumerate}
\item  $y_1$ is unit in a two-dimensional regular local ring $(S_1)_{\frak p}$. 
Also  $\frak p \cap R = \m$ and the image of $y_1$ in the field of fractions of $S_1/\frak p$ is
algebraically independent over $R/\m$. 
\item  $\ord_{(S_1)_{\frak p}}$ is the unique Rees valuation of $\frak p$. To see that this 
valuation is $v$, observe that $v(x) = v(z_1) = 1$ and the image of $\frac{z_1}{x}$ in $k_v$ 
is algebraically independent over the subfield $(S_1)_{\frak p}/ \frak p(S_1)_{\frak p}$ of
$k_v$. This follows because the images of $y_1$ and  $\frac{z_1}{x}$ in $k_v$ are 
algebraically independent over $k$.
\end{enumerate}
\end{proof}

\section{Projectively full finitely supported complete ideals}  \label{s4}

We use the following definitions: 

\begin{definition}\label{4.31}
Let $I$ be a regular proper  ideal in a Noetherian ring $R$.
\begin{enumerate}
\item 
An ideal $J$ in $R$ is {\bf  projectively equivalent} to $I$, 
if some powers of $I$ and $J$ have the same integral closure, 
i.e.,  $\overline{I^j}=\overline{J^i}$ for some $i,j \in \mathbb Z^+$.
\item
The ideal $I$  is said to be {\bf  projective full}, 
if the only complete ideals that are projectively equivalent to $I$ are 
the ideals $\overline{I^k}$ with $k \in \mathbb Z^+$.
\end{enumerate}
\end{definition}

The concept of projective equivalence of ideals was introduced by Samuel in \cite{Sam}
and further developed by Nagata in \cite{Nag}. Making use of work of Rees in \cite{Rees},
McAdam, Ratliff, and Sally in \cite[Corollary~2.4]{MRS} prove that the set $\mathcal P(I)$
of complete ideals projectively equivalent to $I$ is linearly ordered by inclusion and 
discrete.  Moreover, if $I$ and $J$ are projectively equivalent, then $\Rees I = \Rees J$
and the values of $I$ and $J$ with respect to these Rees valuation rings are proportional
\cite[Proposition~ 2.10]{MRS}.  If there exists a projectively full ideal $J$ that is
projectively equivalent to $I$, then the set $\mathcal P(I)$ is said to be
{\bf projectively full}.  As described in \cite{CHRR1}, there is naturally associated to
the projective equivalence class of $I$ a numerical semigroup $\mathcal S(I)$. 
One has $\mathcal S(I) = \mathbb N$, the semigroup of nonnegative integers under addition, 
 if and only if $\mathcal P(I)$ is projectively full.

\begin{remark} \label{4.311} Let $I$ be a finitely supported complete ideal of a 
$d$-dimensional  regular local ring $R$, where $d \geq 2$.
\begin{enumerate}
\item
Every ideal projectively equivalent to $I$ is finitely supported.
\item
If an ideal $J$ is projectively equivalent to $I$, then  
the point bases $\mathcal B(I)$ and $\mathcal B(J)$ are proportional; 
indeed, if $\overline{I^n} = \overline{J^m}$ for positive integers $n$ and $m$,
then  $n\mathcal B(I) = m\mathcal B(J)$. In particular, if $I$ and $J$ are 
projectively equivalent, then $I$ and $J$ have the same base points.
\item
If the greatest common divisor (GCD) of the entries in $\mathcal B(I)$ is 1,
then the ideal $I$ is projectively full.
\item
Every special $*$-simple complete ideal is projectively full.
\end{enumerate}
\end{remark} 

\begin{proof}
These statements all follow from \cite[Remark~1.9 and Proposition~1.10]{L}. 
We write out the details for item (3). 
Let $\mathcal{BP}(I)=\{R_0, R_1, \ldots, R_s\}$, where $R_0:=R$ and 
 $\mathcal{B}(I)=\{\ord_{R_i}(I^{R_i})\}_{i=0}^{s}$.
Let $J$ be a  complete ideal  that is projectively equivalent to $I$.
Then $\overline{I^{n}}=\overline{J^{m}}$ for some $n,~ m \in \mathbb Z^+$.
By \cite[Proposition~1.10]{L}, we have 
 $n\mathcal{B}(I)=\mathcal{B}(I^{n})=\mathcal{B}(J^{m})=m\mathcal{B}(J).$
Let $a_i:=\ord_{R_i}(I^{R_i})$ for $i=0, \ldots, s$.
Then we have 
$$
n\mathcal{B}(I)=n\{a_0,a_1, \ldots, a_s\}=m\{b_0, b_1, \ldots, b_s\}= m\mathcal{B}(J),
$$
where $b_i  = \ord_{R_i}(J^{R_i})$ for $i=0, 1,\ldots, s$.
Since $GCD \{a_0,a_1,\ldots,a_s\}=1$, we have 
$$
n=n\Big( GCD \{a_0,a_1,\ldots,a_s\}\Big) =GCD \{na_0,na_1, \ldots, na_s\} =m\Big( GCD \{b_0, b_1, \ldots, b_s\}\Big).
$$
Hence $n=m r$, where $r = GCD\{b_0, b_1, \ldots, b_s\}$,  and therefore applying 
\cite[Remark~1.9 and Proposition ~1.10]{L},  we have
$$
\aligned
 \overline{I^{mr}}=\overline{J^{m}} &\Longrightarrow  \mathcal{B}(I^{mr})=\mathcal{B}(J^{m}) \\
                                         &\iff mr \mathcal{B}(I)=m \mathcal{B}(J)    \\
                                          &\iff  r \mathcal{B}(I)= \mathcal{B}(J)\\
                                          &\iff  \mathcal{B}(I^r)= \mathcal{B}(J)  \\
                                          &\iff \overline{I^r}=\overline{J} \\
                                          &\iff  \overline{I^r}=J, \qquad \text{since $J$ is complete}.
\endaligned
$$
Thus $I$ is  projective full. 
Item (4) is immediate from item (3), 
because the last nonzero  entry in the point basis of a special $*$-simple complete ideal is $1$. 
\end{proof}

\begin{remark} \label{4.3111} 
Let $(\alpha=R, \m)$ be a $d$-dimensional regular local ring with $d \ge 2$, and let $I$ be a
complete finitely supported $\m$-primary ideal.  Let $\mathcal B(I) = \{a_0,a_1,\ldots,a_s\}$ be
the point basis of $I$.   By Lipman's unique 
factorization thoerem: there exists a unique factorization as in Equation \ref{es30}
$$
\Big( \prod_{n_\delta < 0} P_{\alpha\delta}^{-n_\delta} \Big)*I ~=~ 
\prod_{n_\gamma > 0}P_{\alpha\gamma}^{n_\gamma}
$$
If  $d := GCD\{a_0,a_1,\ldots,a_s\}$, the proof of this unique factorization 
implies that each exponent $n_{\delta}$ and 
$n_{\gamma}$ is a multiple of $d$.  In particular, if there are no negative exponents 
in this factorization, then there exists an ideal $K$ such that $\overline{K^d} = I$. 

In the two-dimensional case,    the Zariski unique factorization theorem implies 
that  $\mathcal P(I)$  is projectively full,  and  $I$ is projectively  full if and only if the GCD of the entries in the point basis of $I$ is equal to 1.  
\end{remark} 

In the higher dimensional case, we ask:

\begin{question} \label{4.312} Let $I$ be a finitely supported complete ideal in 
a $d$-dimensional RLR.  
\begin{enumerate}
\item If $I$ is projectively full, does it follow that 
the GCD of the entries in the point basis of $I$ is 
equal to 1? 

\item If $I$ is $*$-simple, is $I$ projectively full?
\item 
Is $\mathcal P(I)$ always projectively full?
\end{enumerate}
\end{question}

\begin{remark} \label{4.71}
With the notation as in Setting~\ref{2.2}, it is possible that $I$ is projectively full in $R$, 
while the transform $I_1$ is not projectively full in $S_1$. For example, let $d = 2$ and
$\m = (x, y)R$, and let 
$$
I ~ = ~ (x^2, y)^2\m ~ = ~ (x^5, x^3y, xy^2, y^3)R.
$$
Since $\m$ is a simple factor of $I$, the ideal $I$ is projectively full in $R$, 
cf. \cite[Example~3.2]{CHRR2}. We have $S_1 = R[\frac{y}{x}]$. Let $y_1 = \frac{y}{x}$.
Then $IS_1 = x^3I_1$, where $I_1 = (x^2, xy_1, y_1^2)S_1$. Thus the ideal 
$I_1 = (x, y_1)^2S_1$ is not projectively full.
\end{remark} 

\begin{proposition} \label{4.7}   Let $I$ be a complete $\m$-primary ideal
of a regular local ring $(R,\m)$ of dimension $d \ge 2$. With the notation as 
in Setting~\ref{2.2}, if the transform $I_1$ of $I$ in $S_1$ is
projectively full, then  $I$ is projectively full in $R$. 
\end{proposition} 

\begin{proof} Let $J$ be an ideal in $R$ that is projectively equivalent to $I$,
say $\overline{I^n} = \overline{J^m}$  with $n, m$ positive integers. Assume
that $r = \ord_RI$ and $s = \ord_RJ$. Then $IS_1 = x^rI_1$ and $JS_1 = x^sJ_1$.
Thus taking complete transforms, we have
$$
x^{rn}\overline{I_1^n}~ = ~\overline{I^nS_1}~ = ~ \overline{J^mS_1} ~ = ~ x^{sm}\overline{J_1^m}.
$$
Since neither of the ideals $I_1$ nor  $J_1$ in the UFD $S_1$ 
is  contained in a proper principal ideal
of $S_1$, we have $rn = sm$ and $\overline{I_1^n} = \overline{J_1^m}$. Thus $I_1$ and $J_1$ are
projectively equivalent. Since $I_1$ is projectively full, $n = mt$ for some positive integer 
$t$. It follows that $\overline{I^t} = \overline{J}$.
\end{proof}

\section{The structure of special $*$-simple complete ideals} \label{c4}

\begin{setting} \label{7.1}
We consider the structure of special $*$-simple complete ideals as in Definition~\ref{2.13}.
In the case where $\dim R = 2$ and $R \prec T$, the special $*$-simple complete ideal $P_{RT}$ has
a unique Rees valuation $\ord_T$.  In the higher dimensional case, the ideal $P_{RT}$ has
$\ord_T$ as a Rees valuation and often also has other Rees valuations. We observe in
Proposition~\ref{4.0} that the other Rees valuations of $P_{RT}$ are in the set
$\{\ord_{R_i} \}_{i=0}^{n-1}$, where 

\begin{equation} \label{e2}
R=:R_0 ~\subset~ R_1 ~\subset~ R_2~ \subset \cdots \subset~ R_n=T \quad (n \geq 2),
\end{equation}
where $R_{i+1}$ is a local quadratic transform of $R_i$ for $i=0,1, \ldots, n-1$, and
$\dim R = \dim T$.  The residue field $R_n/\m_n$ of $R_n$ is a finite algebraic
extension of the residue field $R_0/\m_0$ of $R_0$. If $R_0/\m_0 = R_n/\m_n$,  
 we observe in Corollary~\ref{4.11} that
the other Rees valuations of $P_{RT}$ are in the set
$\{\ord_{R_i} \}_{i=0}^{n-2}$.
\end{setting}

\begin{definition} \label{7.11} We say {\bf there is no change of direction} for 
the local quadratic sequence $R_0$ to $R_n$ in Equation~\ref{e2} if there exists 
an element $x \in \m_0$ that is part of a minimal generating set of $\m_n$.  We
say {\bf there is a change of direction} between $R_0$ and $R_n$ if $\m_0 \subseteq \m_n^2$. 
\end{definition} 

\begin{remark} \label{7.2} With notation as in Setting~\ref{7.1},  assume that 
$\dim R = \dim T$, and let $I = P_{R_0R_n}$. 
\begin{enumerate}
\item 
By \cite[Corollary~2.2]{L}, the transform $I^{R_j} = P_{R_jR_n}$ for
all $j$ with $0 \le j \le n$. By Proposition~\ref{2.5}, we have
$\Rees_{R_j} I^{R_j} \subseteq \Rees I$.  Thus for each $j$ with $0 \le j \le n$, we have
$$
\Rees_{R_j} P_{R_jR_n} ~ = ~ \Rees_{R_j} I^{R_j} ~ \subseteq ~ \Rees I,
$$
and the number of Rees valuations of $I$ is greater than or equal to the
number of Rees valuations of $P_{R_jRn}$. 
\item
If  $R_0/\m_0 = R_n/\m_n$, then there is no change of direction in the local quadratic sequence from $R_0$ to $R_n$ 
$ \quad \iff \quad \ord_{R_0}(I)=1 \quad\iff \quad \mathcal{B}(I)=\{1, 1, \ldots, 1,1\}.$
\end{enumerate}
\end{remark}

\begin{proposition}\label{4.0}
Let $(R, \m, k)$ be $d$-dimensional regular local ring, where $d \geq 2$, and let $R \prec T$
with $\dim T = d$. Assume the sequence of local quadratic transforms from $R$ to $T$ is as
in Equation~\ref{e2}. 
Let  $P_{R_0R_n}$ be the associated  special $*$-simple complete $\m$-primary ideal in $R$,
and let $V_i$ denote the valuation ring of $\ord_{R_i}$ for $0\le i \le n$. 
Then we have
$$
 \{ V_n \} ~ \subseteq ~ \Rees P_{R_0R_n} ~ \subseteq  ~
\{V_0,~ V_1, ~\ldots, V_{n-1}, ~V_n\}.
$$
\end{proposition}

\begin{proof}
Let $I:=P_{R_0R_n}$.  Since $I^{R_n} = P_{R_nR_n}$ is the maximal ideal of $R_n$,
we have $\{ V_n \}  \subseteq  \Rees P_{R_0R_n}$. Let  $V \in \Rees I$. 
We use the notation of Setting~\ref{2.2}.
Then $IS_1=x^r I_1$, where $r:=\ord_R (I)$. By \cite[Corollary~(2.2)]{L}, 
$I$ is a finitely supported ideal in $R$ and 
$$\mathcal {BP}(I)=\{R_0, R_1, \ldots, R_n\}.$$  
Hence
 $R_1$ is the only base point of $I$ in the first neighborhood of $R$, and $I_1$ is
contained in  a unique  maximal ideal $N_1$ in $S_1$. Hence $R_1 =(S_1)_{N_1}$ 
and $\m_1:=N_1R_1$.
By Proposition~\ref{2.3}, we have 
$$
\aligned
\Rees P_{R_0R_n}=\Rees I &\subseteq \Rees_{S_1} I_1 \cup \Rees \m \\
                         &=\Rees_{R_1} I_1R_1 \cup \Rees \m \\
                         &=\Rees_{R_1} P_{R_1R_n} \cup \Rees \m.
\endaligned
$$
Let $\m_i$ denote the maximal ideal of $R_i$ for $i=1, \ldots, n$.
Since $I^{R_1} = P_{R_1R_n}$, 
a simple induction argument proves that 
$$
\aligned
\Rees P_{R_0R_n} &\subseteq  \Rees_{R_1} P_{R_1R_n} \cup \Rees \m\\
                 &\subseteq  \Rees_{R_2} P_{R_2R_n} \cup \Rees_{R_1} \m_1 \cup \Rees \m\\
                 &\subseteq \cdots \\
                 &\subseteq  \Rees_{R_n} \m_n \cup \Rees_{R_{n-1}} \m_{n-1}\cup 
\Rees_{R_{n-2}} \m_{n-2}\cup \cdots \cup \Rees_{R_1} \m_1 \cup \Rees \m.
\endaligned
$$
\end{proof}

We describe in Remark~\ref{4.1} the structure  of a special $*$-simple complete ideal 
$P_{R_0R_1}$ in the case where  $R_1/\m_1 = R_0/\m_0$. This case
 always occurs if $R_0/\m_0$ is algebraically closed.

\begin{remark}\label{4.1}
Let $R = R_0$ be a $d$-dimensional regular local ring and let $R_1$ be a local quadratic 
transform of $R$ with $\dim R_1 = d$. Let $P_{R_0R_1}$ be the associated 
special $*$-simple complete ideal of $R_0$.
With notation as in Setting~\ref{2.2}, we may assume that 
$$
S_1~= ~R_0\Big[\frac{\m_0}{x_1}\Big]=R_0\Big[\frac{x_2}{x_1}, \ldots, \frac{x_d}{x_1}\Big] ~ \subset ~ R_1, 
$$
where  $\m_0:=\m=(x_1, \ldots, x_d)R_0$. Then $R_1 =(S_1)_{N_1}$, 
where $N_1$ is a maximal ideal in $S_1$ containing $x_1S_1=\m S_1$.
Assume that  $R_1/\m_1 = R_0/\m_0$.   
\begin{enumerate}
\item
Then $N_1=(x_1, \frac{x_2}{x_1}-a_2, \ldots, \frac{x_d}{x_1}-a_d)S_1$, 
where $a_2, \ldots, a_d \in R_0$. We have 
$$
P_{R_0R_1}=(x_1^2, x_2-a_2x_1, \ldots, x_d-a_dx_1)R_0,
$$
and the ideal $P_{R_0R_1}$ has  unique Rees valuation  $w:=\ord_{R_1}$, where  
$$
w(x_1)=1\quad\text{ and}\quad w(x_i-a_ix_1)=2\quad \text{ for}\quad i=2, \ldots, d,
$$
and the images of $\frac{x_2-a_2x_1}{x_1^2}, \ldots, \frac{x_d-a_dx_1}{x_1^2}$ in $k_w$ 
are algebraically independent over $R_0/\m_0$. Thus  $\Rees P_{R_0R_1}=\Rees_{R_1} \m_1$.
\item $\mathcal{BP}(P_{R_0R_1})=\{R_0, R_1\}$.
\item $\mathcal{B}(P_{R_0R_1})=\{1, 1\}$.
\item 
The ideal $I:=P_{R_0R_1}$ is a normal ideal cf. \cite{Go}.  Hence the Rees algebra
$R[It]$ is a normal domain. Also 
$\frac{R[It]}{Q}\cong (\frac{R}{\m})[T_1, \ldots, T_d]$ is a polynomial ring 
in $d$-variables over $R/\m$,
where $\Min (\m R[It])=\{Q\}$ and $Q=\m R[It]$.
\end{enumerate}
\end{remark}

As a consequence of  Proposition~\ref{4.0} and Remark~\ref{4.1}, we have 

\begin{corollary}\label{4.11}
Let the notation be as in Proposition~\ref{4.0}.
Assume that $R_0/\m_0 = R_n/\m_n$. 
Then we have
$$
 \{ V_n \} ~ \subseteq ~ \Rees P_{R_0R_n} ~ \subseteq  ~
\{V_0,~ V_1, ~\ldots, V_{n-2}, ~V_n\}.
$$
\end{corollary}

With notation as in Setting~\ref{7.1}, we illustrate in Example~\ref{4.2} the structure
of the special $*$-simple complete ideal $I = P_{R_0R_n}$ in the case where $R_0/\m_0 = R_n/\m_n$
and there is no change of direction. We assume $\dim R_0 = 3$. The situation is
similar for $\dim R_0  > 3$.

\begin{example}\label{4.2}
Let $(R, \m_0, k)$ be a $3$-dimensional regular local ring with  maximal ideal $\m_0 =(x,y,z)R$,
and let $a_1, \ldots, a_n$ and $b_1, \ldots, b_n$ be elements in $R$.
Consider the following  finite sequence of  local quadratic transformations 
$$
 R=:R_0 \subset  {^xR_1} \subset  {^{xx}R_2} \subset {^{xxx}R_3} \subset \cdots \subset 
^{\overbrace{x \cdots  x}^{\text{$n$ times}}}R_n,
$$
where for $i=0, 1, 2,\ldots, n-1$ we define $S_{i+1}$ and $
^{\overbrace{x \cdots  x}^{\text{$i+1$ times}}}R_{i+1} = R_{i+1}$ inductively by
$$
\begin{aligned}
& S_1 ~ :=  ~R_0[\frac{\m_0}{x}], \quad N_1 ~:= ~ (x, \frac{y-a_1x}{x}, \frac{z-b_1x}{x})S_1 \quad
R_1 ~:= ~ (S_1)_{N_1} \quad \m_1 ~:= ~ N_1R_1 \\
& \cdots  \\
& S_{i+1} ~:= ~ R_i[\frac{\m_i}{x}], \quad N_{i+1} ~:= ~ 
(x, \frac{y-a_1x - \cdots - a_{i+1}x^{i+1}}{x^{i+1}}, 
\frac{z-b_1x - \cdots - b_{i+1}x^{i+1}}{x^{i+1}})S_{i+1} \\
& R_{i+1} ~:= ~ (S_{i+1})_{N_{i+1}} \quad \m_{i+1} ~:= ~ N_{i+1}R_{i+1}. 
\end{aligned}
$$

 Then for $i=0,1,\ldots,n,$ we have
\begin{enumerate}
\item 
The order valuation $v_i:=\ord_{R_i}$ has values $v_i(x)=1$ and 
$$
v_i\Big(\frac{y-a_1x - \cdots - a_{i}x^{i}}{x^{i}}\Big) \quad ~=  \quad
v_i\Big(\frac{z-b_1x - \cdots - b_{i}x^{i}}{x^{i}}\Big) ~= ~ 1.
$$
and the images of 
$$
\frac{y-a_1x - \cdots - a_{i}x^{i}}{x^{i+1}} \quad \text{and} \quad 
\frac{z-b_1x - \cdots - b_{i}x^{i}}{x^{i+1}}
$$
in the residue field $k_{v_i}$ of $V_i$ are algebraically independent over $R_0/\m_0$.
\item 
The special $*$-simple complete $\m$-primary ideal is  
$$P_{R_0R_i} =(x^{i+1},~ y -a_1x - \cdots - a_ix^i, ~  z - b_1x - \cdots - b_ix^i)R.$$
\item
$\mathcal {BP}(P_{R_0R_i})=\{R_0,  R_1,  R_2, \ldots, R_i\}.$
\item 
$\mathcal{B}(P_{R_0R_i}) =\{1,~ 1,~ 1,~ \ldots,~ 1\}.$
\item 
The special $*$-simple complete $\m$-primary ideal $P_{R_0R_i}$ has a 
unique Rees valuation $\ord_{R_i}$. 
That is, $\Rees (P_{R_0R_i})=\Rees_{R_i} \m_i$.
\item
The ideal $P_{R_0R_i}$ is normal.
\item 
Let $I:=P_{R_0R_i}$. Then $\m R[It]$ has a unique minimal prime $Q:=\m R[It]$ 
and $\frac{R[It]}{Q}$ is a polynomial ring in $3$-variables over $R/\m$.
\end{enumerate}
\end{example}

\begin{theorem}\label{4.40}
Let the notation be as in  Setting~\ref{7.1}. Assume that $R_0/\m_0 = R_n/\m_n$. 
Then the following are equivalent: 
\begin{enumerate}
\item $\Rees P_{R_0R_n}=\ \Rees_{R_n} \m_n$, i.e., $\ord_{R_n}$ is the unique Rees valuation of
$P_{R_0R_n}$. 
\item 
Either $\dim R_0 = 2$, or there is no change of direction in the local quadratic  sequence 
given in Equation~\ref{e2}.
\end{enumerate} 
\end{theorem}

\begin{proof}
$(2) \implies (1)$:  If $\dim R_0 = 2$, then the theory of Zariski implies that $I$ has a 
unique Rees valuation. Assume that $\dim R_0 \ge 3$ and 
that there is no change of direction in the  sequence given in Equation~\ref{e2}. 
By Example~\ref{4.2}, 
we have that the special $*$-simple complete ideal $P_{R_0R_n}$ has the unique Rees valuation, $\ord_{R_n}$.\\
$(1) \implies (2)$: First, notice that $\Rees_{R_n} \m_n~\subseteq \Rees P_{R_0R_n}$ 
by Proposition~\ref{4.0}, 
and hence $\vert \Rees P_{R_0R_n}\vert \geq~1$.
To conclude the proof, we prove the following :

\begin{claim}\label{4.401}
If there is at least one change of direction in the local quadratic sequence 
given in Equation~\ref{e2}, then
$\vert \Rees P_{R_0R_n}\vert >~1$.
\end{claim} 

\begin{proof}
Assume there is at least one change of direction between $R_0$ and $R_n$. Choose $j$ minimal 
so that there is no change of direction from $R_{j+1}$ to $R_n$. Then by 
choosing  appropriate  regular parameters $x$ in $R_j$ and $y$ in $R_{j+1}$, we have 
the following local quadratic sequence:
$$
R=:R_0 \subset  {R_1} \subset  {R_2} 
\subset \cdots \underset{\overset \shortparallel {A_{0}}} {R_{j}} 
\subset \underset{\overset \shortparallel {^{x}A_{1}}} {R_{j+1}} 
\subset \underset{\overset \shortparallel {^{yx}A_{2}}} {R_{j+2}} 
\subset \underset{\overset \shortparallel {^{yyx}A_{3}}} {R_{j+3}} 
\subset \cdots \subset \underset{\overset \shortparallel {^{\overbrace{y\cdots yx}^{\text{$n-j$ times}}}A_{n-j}}} {R_{n}}.
$$
By Remark~\ref{7.2}, we have $\Rees_{R_j} P_{R_jR_n} \subseteq \Rees P_{R_0R_n}$. 
Thus to complete the proof of the Claim, we analyse in Example~\ref{4.300}
the structure of a special $*$-simple complete $\m$-primary ideal of a
$d \ge 3$-dimensional regular local ring 
obtained by a change of direction first dividing by 
$x$ and then successively by $y$.  For notational simplicity, we assume that $d = 3$. The 
pattern is similar in the case where $d > 3$. 
\end{proof}

\begin{example}\label{4.300}
Let $(R, \m, k)$ be a $3$-dimensional regular local ring with  maximal 
ideal $\m=(x,y,z)R$. Let $n \geq 3$ .
Consider a  sequence of  local quadratic transforms 
 $$
R~:= ~R_0 ~\subset ~  R_1 ~:= ~  {^xR_1}~ \subset ~R_2 ~:=~ {^{yx}R_2} ~\subset ~R_3  ~:=~ {^{yyx}R_3}~\subset~\cdots ~
\subset~R_n~:={^{\overbrace{y\cdots yyx}^{\text{$n$ times}}}}{R_n}
$$
defined by
$$
\begin{aligned}
& S_1 := R[\frac{\m}{x}],  & N_1&:= (x, \frac{y}{x}, \frac{z}{x})S_1,
&R_1 :=  (S_1)_{N_1},  &\m_1 := N_1R_1 \\
& S_{2} :=R_1[\frac{\m_1}{y/x}], & N_{2}&:=(\frac{x^2}{y}, \frac{y}{x}, \frac{z- b_2y}{y} )S_2,
&R_2:=  (S_{2})_{N_{2}}, &\m_{2} :=  N_{2}R_{2} \\
& S_{3} := R_2[\frac{\m_2}{y/x}],  & N_{3}&:= (\frac{x^3-a_3y^2}{y^2}, \frac{y}{x}, \frac{xz- b_2xy -b_3y^2}{y^2} )S_3,
&R_3:=  (S_{3})_{N_{3}},  &\m_{3} := N_{3}R_{3}\\
&\cdots \\
& S_{n} := R_{n-1}[\frac{\m_{n-1}}{y/x}],
& N_{n} &:= (f,~ g,~h)S_n, 
&R_n:=  (S_{n})_{N_{n}}, &\m_{n} :=  N_{n}R_{n},
\end{aligned}
$$
where 
$$
\aligned 
f~&:=~\frac{x^n-a_3x^{n-3}y^2-\cdots-a_{n-1}xy^{n-2}-a_ny^{n-1}}{y^{n-1}}\\
g~&:= ~\frac{y}{x}\\ 
h~&:=~\frac{x^{n-2}z - b_2x^{n-2}y -\cdots-b_{n-1}xy^{n-2}-b_ny^{n-1}}{y^{n-1}} .
\endaligned
$$
Here the elements $a_i$ and $b_j$ are in $R_0$, and 
we are assuming that $R_0/\m_0 = R_n/\m_n$.  Thus we may choose $x,y,z$ so that 
$N_1 = (x, \frac{y}{x}, \frac{z}{x})S_1$.  We are also assuming that there is a 
change of direction from $R_0$ to $R_2$. Thus we may assume 
$N_{2}=(\frac{x^2}{y}, \frac{y}{x}, \frac{z- b_2y}{y} )S_2$. 

Let 
$$
\aligned
f_0~&:=x^n-a_3x^{n-3}y^2-\cdots-a_{n-1}xy^{n-2}-a_ny^{n-1}\\
 h_0~&:=x^{n-2}z - b_2x^{n-2}y -\cdots-b_{n-1}xy^{n-2}-b_ny^{n-1}.
\endaligned
$$
Then: 
\begin{enumerate}
\item
 Let $v_n:=\ord_{R_n}$. We have 
\begin{equation}\label{e200}
\aligned
v_n(f_0)&=1+(n-1)v_n(y),\\
v_n(y)&=1+v_n(x),\\
v_n(h_0)&=1+(n-1)v_n(y).
\endaligned
\end{equation}
\item
Let 
$$
K ~:= ~ \{ \alpha \in \m_0 ~|~ v_n(\alpha) ~\ge ~v_n(y^n)  ~~ \text{and} ~~ v_0(\alpha) ~ \ge ~ n \}.
$$
Then 
we have  $K=P_{R_0R_n}$
\item 
$\mathcal{BP}(P_{R_0R_n})=\{R_0, ~R_1,~R_2,~R_3~\ldots,~R_n\}$.
\item 
$\mathcal{B}(P_{R_0R_n})=\{n,~1,~1,~\ldots,~1,~1\}$.
\item 
The Rees valuations of $P_{R_0R_n}$ are $\ord_{R_0}$ and $\ord_{R_n}$.
\end{enumerate}
\end{example}

\begin{proof}
 For item (1), since $\m_n = (f, ~g,~h)R_n$ 
we have
 $$
\aligned
1~&=~v_n(~\frac{x^n-a_3x^{n-3}y^2-\cdots-a_{n-1}xy^{n-2}-a_ny^{n-1}}{y^{n-1}})~=~v_n(\frac{y}{x})\\
  &=~v_n(\frac{x^{n-2}z - b_2x^{n-2}y -\cdots-b_{n-1}xy^{n-2}-b_ny^{n-1}}{y^{n-1}} ),
\endaligned
$$
and hence 
$$
v_n(f_0)=1+(n-1)v_n(y),\quad v_n(y)=1+v_n(x),\quad v_n(h_0)=1+(n-1)v_n(y).
$$
Multiplying  the listed generators of $\m_n$ by $xy^{n-1}$,  we obtain elememts in $R_0$
$$
 x\cdot f_0,\quad y^n,\quad x\cdot h_0
$$
By Equation~\ref{e200}, we have
$$
n(1+v_n(x))=v_n(x f_0)=v_n(y^n)=v_n( x h_0).
$$
For item (2), we clearly have  $K \supseteq (xf_0,~ y^n,~xh_0)R_0$.  \\ 
We observe that $v_n(z) \geq v_n(y)$;  for if $v_n(z) < v_n(y)$, 
then 
$$
v_n(x^{n-1}z) ~ < ~  v_n(-b_2x^{n-1}y - \cdots - b_nxy^{n-1})
$$
implies that $v_n(x^{n-1}z) = v_n(xh_0) = v_n(y^n)$, a contradiction. Therefore $z^n \in K$.

Since $K$ has order $n$ and contains $y^n$ and $z^n$, we see that the transform of $K$
in $R_0[\frac{\m_0}{y}]$ and in $R_0[\frac{\m_0}{z}]$ is the whole ring. 
Let $y_1:=\frac{y}{x}$ and $z_1:=\frac{z}{x}$, then $y=xy_1$ and $z=xz_1$. Hence
$$
\aligned
& KR_0[\frac{\m_0}{x}] = KS_1 \\
&\supseteq x^n\big(~x-(a_3y_1^2-\cdots-a_ny_1^{n-1}),~y_1^n,~z_1-(b_2y_1-b_3y_1^2-\cdots-b_ny_1^{n-1}~\big)S_1.
\endaligned
$$
Thus the transform of $K$ in $S_1$ is
$$
K^{S_1} \supseteq \big(~x-(a_3y_1^2-\cdots-a_ny_1^{n-1}),~y_1^n,~z_1-(b_2y_1-b_3y_1^2-\cdots-b_ny_1^{n-1})~\big)S_1.
$$
Since  the ideal $K^{S_1}$ is primary for the maximal ideal $N_1 = (x_1,y_1,z_1)S_1$
and $R_1 = (S_1)_{N_1}$, we have 
$$
K^{S_1}=\big(~x-(a_3y_1^2-\cdots-a_ny_1^{n-1}),~y_1^n,~z_1-(b_2y_1-b_3y_1^2-\cdots-b_ny_1^{n-1})\big)S_1, 
$$
and hence 
$$
K^{R_1}=\big(~x-(a_3y_1^2-\cdots-a_ny_1^{n-1}),~y_1^n,~z_1-(b_2y_1-b_3y_1^2-\cdots-b_ny_1^{n-1})\big)R_1. 
$$
As in Example~\ref{4.2}, we have 
$$
P_{R_1R_n}= \big(~x-(a_3y_1^2 - \cdots -a_ny_1^{n-1}),~y_1^n,~z_1-(b_2y_1-b_3y_1^2-\cdots-b_ny_1^{n-1})\big)R_1, 
$$
and thus $K^{R_1} = P_{R_1R_n}$. 
By \cite[Proposition~2.1]{L}, we have $K=P_{R_0R_n}$.
Items (3) and  (4)  are clear. Since $K$ has order $n$ and contains $y^n, z^n$ and $xh_0$, we
see that $\ord_{R_0}$ is a Rees valuation of $K$. Therefore $\ord_{R_0}$ and $\ord_{R_n}$ are
the Rees valuations of $K$. 
\end{proof}

This completes the proof of Theorem~\ref{4.40}.
\end{proof}

We illustrate in Examples~\ref{4.121} and \ref{4.122} the behavior of 
 a special $*$-simple complete ideal $P_{R_0R_1}$ in cases where $[R_1/\m_1 : R_0/\m_0]>1$.
In Example~\ref{4.121} the ideal $P_{R_0R_1}$ has two Rees valuations,  while in 
Example~\ref{4.122} the ideal $P_{R_0R_1}$ has only one Rees valuation. 
We use notation as in Remark~\ref{4.1} with $d = 3$ and $\m_0 = (x,y,z)R_0$.

\begin{example}\label{4.121}
Let $R_0/\m_0=\mathbb Q$ and  $R_1:=(S_1)_{N_1}$, where 
$$
N_1:=(x, \frac{y}{x}, \big(\frac{z}{x}\big)^2-3)S_1.
$$
Let $w:=\ord_{R_1}$. Then we have
\begin{enumerate}
\item
$$
w(x)=w(\frac{y}{x})=w(\big(\frac{z}{x}\big)^{2}-3)=1,
$$
and the images of $\frac{y}{x^2},  ~ \frac{z^2 - 3x^2}{x^3}$ in the 
residue field $k_w$ of $w$ are 
algebraically independent over $R_0/\m_0$. Also 
$w(z^2-3x^2)=1+w(x^2)=3$. Therefore
\[
\begin{array}{c|c|c|c | c}
        & x & y & z    & z^2 - 3x^2      \\ \hline
w:=\ord_{R_1}      & 1   & 2    & 1  & 3                           
\end{array}
\]
\item  Let 
$$
I:=\{\alpha \in \m_0 ~\vert~ w(\alpha) \geq 3 \}.
$$
Then we have
\begin{enumerate}
\item $I= (x^3, xy, z^2-3x^2, ~y^2, yz, z^3)R_0$.  A 
direct computation shows that $I  =P_{R_0R_1}$. We have 
\[
\begin{array}{c|c|c|c|c|c|c}
     P_{R_0R_1}   & x^3 & xy   &  z^2-3x^2  & y^2 & yz & z^3         \\ \hline
w:=\ord_{R_1}      & 3       & 3     & 3            &4 & 3                 &3               \\ \hline           
  v:=\ord_{R_0}      &3    &2     &2      &2             & 2    & 3                  
\end{array}
\]
\item $\mathcal{BP}(P_{R_0R_1})=\{R_0, R_1 \}$.
\item $\mathcal{B}(P_{R_0R_1})=\{2, 1\}$.
\item The set of Rees valuations of $P_{R_0R_1}$ is $\{\ord_{R_0}, \ord_{R_1}\}$.
\end{enumerate}
\end{enumerate}
\end{example}

\begin{example}\label{4.122}
Let $R_0/\m_0=\mathbb Q$ and  $R_1:=(S_1)_{N_1}$, where 
$$
N_1:=(x, ~ \big(\frac{y}{x}\big)^2-2, ~ \big(\frac{z}{x}\big)^2-3)S_1.
$$
Let $w:=\ord_{R_1}$. Then we have
\begin{enumerate}
\item
$$
w(x)=w(\big(\frac{y}{x}\big)^2-2)=w(\big(\frac{z}{x}\big)^{2}-3)=1,
$$
and  the images of $\frac{y^2 - 2x^2}{x^3}, \frac{z^2 - 3x^2}{x^3}$ in 
the residue field $k_w$ of $w$ are algebraically independent over $R_0/\m_0$.  Also 
$w(y^2-2x^2)=w(z^2-3x^2)=1+w(x^2)=3$. Therefore
\[
\begin{array}{c|c|c|c}
        & x & y & z          \\ \hline
w:=\ord_{R_1}      & 1   & 1    & 1                            
\end{array}
\]
\item  Let 
$$
I:=\{\alpha \in \m_0 ~\vert~ w(\alpha) \geq 3\}.
$$
Then we have
\begin{enumerate}
\item $I=(~y^2-2x^2,~z^2-3x^2,~ \m_0^3~)R_0$.  A direct computation
shows that $I  =P_{R_0R_1}$.  We have 
\[
\begin{array}{c|c|c|c|c|c|c}
     P_{R_0R_1}   & ~y^2-2x^2   &  z^2-3x^2  & \m_0^3         \\ \hline
w:=\ord_{R_1}      & 3       & 3     & 3                      \\ \hline           
  v:=\ord_{R_0}      &2    &2     & 3                      
\end{array}
\]
\item $\mathcal{BP}(P_{R_0R_1})=\{R_0, R_1 \}$.
\item $\mathcal{B}(P_{R_0R_1})=\{2, 1\}$.
\item $\Rees (P_{R_0R_1})= \Rees_{R_1}\m_1 $.
\end{enumerate}
\end{enumerate}
\end{example}

 Example~\ref{4.3} illustrates a pattern with exactly one  change of direction from $R_0$ to $R_2$
where $R_0/\m_0 = R_2/\m_2$.  

\begin{example}\label{4.3}
Let $(R, \m, k)$ be a $3$-dimensional regular local ring with  maximal 
ideal $\m=(x,y,z)R$.
Consider the following sequence of  local quadratic transforms 
 $$
R~:= ~R_0 ~\subset ~  R_1 ~:= ~  {^xR_1}~ \subset ~R_2 ~:=~ {^{yx}R_2} 
$$
defined by
$$
\begin{aligned}
& S_1 ~ :=  ~R[\frac{\m}{x}], \quad  & N_1 ~&:= ~ (x, \frac{y}{x}, 
\frac{z}{x})S_1,  \quad
&R_1 ~:= ~ (S_1)_{N_1},  \quad  &\m_1 ~:= ~ N_1R_1 \\
& S_{2} ~:= ~ R_1[\frac{\m_1}{y/x}], \quad  & N_{2} ~&:= ~ 
 (\frac{x^2}{y}, \frac{y}{x}, \frac{z - b_2y}{y} )S_2,  \quad 
&R_2~:= ~ (S_{2})_{N_{2}}, \quad &\m_{2} ~:= ~ N_{2}R_{2} 
\end{aligned}
$$
Then: 
\begin{enumerate}
\item 
Let $v_2:=\ord_{R_2}$. Then $v_2(x)=2$, $v_2(y)=3$, $v_2(z - b_2y)=  4$,
and the images of 
$\frac{x^3}{y^2}, \frac{x(z - b_2y)}{y^2} $ 
in the residue field $k_{v_2}$ of $V_2$ are algebraically independent over $R_2/\m_2 = k$.
\item 
The special $*$-simple complete $\m$-primary ideal $P_{R_0R_2}$ is  a $v_2$-ideal. We have
$$
\aligned
P_{R_0R_2} ~ & = ~ \{ \alpha \in \m ~|~ v_2(\alpha) \ge 6 \} \\
 ~&=~ (x^3,~ x(z - b_2y), ~ y^2,~ x^2y, ~ y(z- b_2y),  ~ (z - b_2y)^2)R.
\endaligned
$$
\item
$\mathcal {BP}(P_{R_0R_2})=\{R_0,~  R_1, ~ R_2\}.$
\item 
$\mathcal{B}(P_{R_0R_2}) =\{2,~ 1,~ 1\}.$
\item  
The set of Rees valuations of $P_{R_0R_2}$ is $\{\ord_{R_0},~ \ord_{R_2}\}$.  
\item 
Let $I:=P_{R_0R_2}$. Then :
\begin{enumerate}
\item $\Min (\m R[It])=\{P_0,P_2\}$, where 
$$
\aligned
P_2 &=(\m ,~ x^2yt, ~ y(z-b_2y)t,~ (z-b_2y)^2t)R[It]\quad \text{ and}\\
 P_0 &=(\m , ~x^3t, ~ x^2yt )R[It].
\endaligned
$$
\item
Let $V_0$ and $V_2$ denote the valuation rings corresponding to $v_0:=\ord_{R_0}$ 
and $v_2:=\ord_{R_2}$,
 where $V_0=R[It]_{P_0}\cap \mathcal{Q}(R)$ and $V_2=R[It]_{P_2}\cap \mathcal{Q}(R)$. 
Then
$\frac{R[It]}{P_2}$ is a polynomial ring in $3$-variables over $R/\m$, and 
$\frac{R[It]}{P_0}\cong \frac{(R/\m)[T_1, T_2, T_3, T_4]}{(T_2T_4-T_3^2)})$ is a 
$3$-dimensional normal Cohen-Macaulay domain
with minimal multiplicity at its maximal homogeneous ideal with this multiplicity being $2$. 
\end{enumerate}
\end{enumerate}
\end{example}

\begin{proof}
(1) :  Since $\m_2=(\frac{x^2}{y}, ~\frac{y}{x},~ \frac{z-b_2y}{y})R_2$, we have $v_2(\frac{x^2}{y})=v_2(\frac{y}{x})=v_2(\frac{z-b_2y}{y})=1$, 
and hence  $v_2(x)=2, v_2(y)=3, v_2(z - b_2y)=  4$, and  
and   the images of $\frac{x^3}{y^2}, \frac{x(z-b_2y)}{y^2}$ in the residue field $k_{v_2}$ of $V_2$ are algebraically independent over $R_2/\m_2 = k$.\\
(2), (3), and (4) : The transform in $R_2$ of the ideal $K:=(x^3, y^2, x(z-b_2y))R$ is 
the maximal ideal $\m_2=(\frac{x^2}{y}, \frac{y}{x}, \frac{z-b_2y}{y})R_2$ and $v_2(K) = 6$. 
Let 
$$
I~:= ~\{\alpha \in \m ~\vert~v_2(\alpha) \geq 6\} ~ = ~(x^3, x(z-b_2y), y^2, x^2y, 
y(z - b_2y), (z-b_2y)^2)R.
$$ 
The ideal $I$ is a complete $\m$-primary ideal. We see by direct computation
that:
\begin{enumerate}
\item[(a)] $I^{R_1}=(x, (\frac{y}{x})^2, \frac{z-b_2y}{x})R_1$ is a special $*$-simple 
complete  ideal in $R_1$.
\item[(b)]
$I^{R_2}=\m_2$. 
\item[(c)]   
$\mathcal{BP}(I)=\{R_0, R_1,  R_2\}$. 
\end{enumerate}
Thus by \cite[Proposition~2.1]{L}, 
we have $I=P_{R_0R_2}$.  It is clear that 
$\ord_{R_0}(I)=2, \ord_{R_1}(I^{R_1})=1, \ord_{R_2}(I^{R_2})=1.$
Hence $\mathcal{B}(I)=\{2,~1,~1\}$.\\
(5) :  By item (1), $v_2:=\ord_{R_2}$ is a Rees valuation of  $P_{R_0R_2}$. 
We have the following table :
\[
\begin{array}{c|c|c|c|c|c|c|c|c}
   P_{R_0R_2}       & x^3 & x(z-b_2y)  & y^2   &x^2y   &y(z-b_2y)   &(z-b_2y)^2         \\ \hline
v_2:=\ord_{R_2}      & 6   & 6    & 6    &7     &7     &8               \\ \hline
v_1:=\ord_{R_1}         & 3   & 3    &4     &4     &4     &4              \\ \hline
v_0:=\ord_{R_0}      &3    &2     &2       &3    &2    &2             
\end{array}
\]
Since the images of $\frac{y^2}{yz}, \frac{x(z-b_2y)}{yz}$ in the residue field $k_{v_0}$ of $V_0$ are algebrically independent over $k$, 
  $v_0:=\ord_{R_0}$  is a Rees valuation of  $P_{R_0R_2}$. Using 
Proposition~\ref{4.0}, we conclude that $P_{R_0R_2}$ has the  two Rees valuations, $v_0:=\ord_{R_0} $ and $v_2:=\ord_{R_2}$. \\
(6) :  The statements of item (6) follow from the previous items and the connection between 
the Rees valuations of $I$ and minimal primes of $\m R[It]$ in the Rees algebra $R[It]$, 
cf. \cite{HK}.
\end{proof}

 Example~\ref{4.4} illustrates a pattern where there are exactly two changes of direction from
$R_0$ to $R_3$ and where $R_0/\m_0=R_3/\m_3$.

\begin{example}\label{4.4}
Let $(R, \m, k)$ be a $3$-dimensional regular local ring with  maximal ideal $\m=(x,y,z)R$.
Consider the following sequence of  local quadratic transforms 
 $$
R~:= ~R_0 ~\subset ~  R_1 ~:= ~  {^xR_1}~ \subset ~R_2 ~:=~ {^{yx}R_2}~ \subset ~R_3 ~:=~ {^{zyx}R_3}
$$
defined by
$$
\begin{aligned}
& S_1 ~ :=  ~R[\frac{\m}{x}], \quad  & N_1 ~&:= ~ (x, \frac{y}{x}, \frac{z}{x})S_1,  \quad
&R_1 ~:= ~ (S_1)_{N_1},  \quad  \m_1 ~:= ~ N_1R_1 \\
& S_{2} ~:= ~ R_1[\frac{\m_1}{y/x}], \quad  & N_{2} ~&:= ~ 
 (\frac{x^2}{y}, \frac{y}{x}, \frac{z}{y})S_2,  \quad 
&R_2~:= ~ (S_{2})_{N_{2}}, \quad \m_{2} ~:= ~ N_{2}R_{2} \\ 
& S_{3} ~:= ~ R_2[\frac{\m_2}{z/y}], \quad  & N_{3} ~&:= ~ 
 (\frac{x^2}{z}, \frac{y^2}{xz}, \frac{z}{y})S_2,  \quad 
&R_3~:= ~ (S_{3})_{N_{3}}, \quad \m_{3} ~:= ~ N_{3}R_{3}. 
\end{aligned}
$$
Then: 
\begin{enumerate}
\item 
Let $w_3:=\ord_{R_3}$. Then $w_3(x)=4$, $w_3(y)=6$, $w_3(z)=7$, and the images of 
$\frac{x^2y}{z^2}, \frac{y^3}{xz^2}$ 
in the residue field $k_{w_3}$ of $W_3$ are algebraically independent over $R_3/\m_3 = k$.
\item 
The special $*$-simple complete $\m$-primary ideal $P_{R_0R_3}$ is  a $w_3$-ideal. We have
$$
\aligned
P_{R_0R_3} ~ &= ~ \{ \alpha \in \m ~|~ w_3(\alpha) \ge 18 \}\\
            &=~(x^3y,~ y^3, ~ xz^2,~ y^2z, ~ x^3z,~ x^5, ~yz^2,~x^2y^2, z^3,~x^2yz)R.
\endaligned
$$
\item
$\mathcal {BP}(P_{R_0R_3})=\{R_0,~  R_1, ~ R_2, ~ R_3\}.$
\item 
$\mathcal{B}(P_{R_0R_3}) =\{3,~2, ~ 1,~ 1\}.$
\item  
The set of Rees valuations of $P_{R_0R_3}$ is $\{\ord_{R_0},~\ord_{R_1}, ~ \ord_{R_3}\}$.  
\item 
Let $I:=P_{R_0R_3}$. Then :
\begin{enumerate}
\item $\Min (\m R[It])=\{Q_0, Q_1, Q_3\}$, where 
$$
\aligned
Q_3= &(\m ,~ y^2zt, ~x^3zt,~ x^5t,~ yz^2t,~ x^2y^2t, ~z^3t,~ x^2yzt)R[It] \\  
Q_1= &(\m , ~y^3t, ~ y^2zt,~ yz^2t,~ x^2y^2t,~ z^3t,~x^2yzt)R[It] \\
                     Q_0=&(\m,~ x^3yt,~ x^3zt,~ x^5t,~ x^2y^2t,~ x^2yzt)R[It].
\endaligned
$$
\item
Let $W_i$ denote the valuation rings corresponding to $w_i:=\ord_{R_i}$,
 where $W_i=R[It]_{Q_i}\cap \mathcal{Q}(R)$ for $i=0,1,3$. Then
$\frac{R[It]}{Q_3}$ is a polynomial ring in $3$-variables over $R/\m$, and 
$\frac{R[It]}{Q_1}\cong \frac{(R/\m)[T_1, T_2, T_3, T_4]}{(T_2T_4-T_3^2)}$ is a 
$3$-dimensional normal Cohen-Macaulay domain
with minimal multiplicity at its maximal homogeneous ideal with this multiplicity being $2$.  
$\frac{R[It]}{Q_0}\cong \frac{(R/\m)[T_1, T_2, T_3, T_4, T_5]}{\mathcal J}$ is a 
$3$-dimensional normal Cohen-Macaulay domain
with minimal multiplicity at its maximal homogeneous ideal with this multiplicity being $3$, 
where the ideal $\mathcal J$ is generated by the $2 \times 2$ minors of the matrix
$\left[\begin{matrix}
                        T_1 &T_3& T_4\\
                       T_3&T_4&T_5 
                    \end{matrix} \right]$. 
\end{enumerate}
\end{enumerate}
\end{example}

\begin{proof}
(1) :  Since $\m_3=(\frac{x^2}{z},~ \frac{y^2}{xz},~ \frac{z}{y})R_3$, 
we have $w_3(\frac{x^2}{z})=w_3(\frac{y^2}{xz})=w_3(\frac{z}{y})=1$, 
and hence  $w_3(x)=4, w_3(y)=6, w_3(z)=7$, 
and   the images of   
$\frac{x^2y}{z^2}, \frac{y^3}{xz^2}$ 
in the residue field $k_{w_3}$ of $W_3$ are algebraically independent over $R_3/\m_3 = k$.\\
(2), (3), and (4) : The transform in $R_3$ of the ideal $K:=(x^3y, y^3, xz^2)R$ is 
the maximal ideal $\m_3=(\frac{x^2}{z}, \frac{y^2}{xz}, \frac{z}{y})R_3$ and $w_3(K) = 18$. 
Let 
$$
I~:= ~\{\alpha \in \m ~\vert~w_3(\alpha) \geq 18\} ~ = ~ (x^3y,~ y^3, ~ xz^2,~ y^2z, ~ x^3z,~ x^5, ~yz^2,~x^2y^2, z^3,~x^2yz)R.
$$ 
The ideal $I$ is a $w_3$-ideal. We see by direct computation
that:
\begin{enumerate}
\item[(a)] $I^{R_1}= \big((\frac{y}{x})^3, ~y, ~(\frac{z}{x})^2,~(\frac{y}{x})^2(\frac{z}{x}),~z, ~x^2 \big)R_1 = P_{R_1R_3}$.
\item[(b)]  $I^{R_2} = K^{R_2} = (\frac{y}{x}, \frac{x^2}{y}, (\frac{z}{y})^2)R_2  = P_{R_2R_3}$
\item[(c)]
$I^{R_3}=\m_3$. 
\item[(d)]   
$\mathcal{BP}(I)=\{R_0, R_1,  R_2, R_3\}$. 
\end{enumerate}
Thus by \cite[Proposition~2.1]{L}, 
we have $I=P_{R_0R_3}$.  It is clear that 
$\ord_{R_0}(I)=3, \ord_{R_1}(I^{R_1})=2, \ord_{R_2}(I^{R_2})=1, \ord_{R_3}(I^{R_3})=1.$
Hence $\mathcal{B}(I)=\{3, ~2,~1,~1\}$.\\
(5) :  By item (1), $w_3:=\ord_{R_3}$ is a Rees valuation of  $P_{R_0R_3}$. 
We have the following table :
\[
\begin{array}{c|c|c|c|c|c|c|c|c|c|c|c|c|c}
   P_{R_0R_3}       & x^3y & y^3  &  xz^2  & y^2z  &  x^3z  &  x^5  &  yz^2  & x^2y^2  & z^3  & x^2yz  \\ \hline
w_3:=\ord_{R_3}      & 18   & 18   & 18   &19    &19    &20   &20   &20   &21   &21              \\ \hline
w_2:=\ord_{R_2}       & 9    & 9    &10    &10    &10    &10   &11   &10   &12   &11              \\ \hline
w_1:=\ord_{R_1}      &5     &6     &5     &6     &5     &5    &6    &6    &6    &6                 \\ \hline
w_0:=\ord_{R_0}     &4     &3    &3     &3     &4     &5    &3    &4    &3    &4       
\end{array}
\]
Since the images of $\frac{x^3z}{x^3y}, \frac{x^3z}{x^5}$ in the residue field $k_{w_1}$ of $W_1$ are algebrically independent over $k$, 
  $w_1:=\ord_{R_1}$  is a Rees valuation of  $P_{R_0R_3}$, and also
 since the images of $\frac{xz^2}{z^3}, \frac{y^2z}{y^3}$ in the residue field $k_{w_0}$ of $W_0$ are algebrically independent over $k$, 
  $w_0:=\ord_{R_0}$  is a Rees valuation of  $P_{R_0R_3}$.
Using 
Proposition~\ref{4.0}, we conclude that $P_{R_0R_3}$ has the  three Rees valuations, $\ord_{R_0}, \ord_{R_1}$ , and $\ord_{R_3}$. \\
(6) :  The statements of item (6) follow from the previous items  and the connection between 
the Rees valuations of $I$ and minimal primes of $\m R[It]$ in the Rees algebra $R[It]$, 
cf. \cite{HK}.
\end{proof}

\end{document}